\documentclass[amstex,12pt]{article}
\usepackage{mathrsfs}
\usepackage{amsfonts}
\usepackage{amssymb,amsmath,amsthm}
\newtheorem{theorem}{Theorem}[section]

\newtheorem{lemma}{Lemma}[section]

\newtheorem{definition}{Definition}[section]
\newtheorem{remark}{Remark}[section]

\newcommand{\beq}{\begin{equation}}
\newcommand{\eeq}{\end{equation}}
\newcommand{\beqn}{\begin{eqnarray}}
\newcommand{\eeqn}{\end{eqnarray}}

\newcommand{\DF}[2]{{\displaystyle\frac{#1}{#2}}}

\def\a{{\alpha}}

\def\e{{\varepsilon}}

\begin{document}

\title{Almost periodic functions on the quantum time scale and applications\thanks{This work is supported by the National Natural Sciences Foundation of People's Republic of China under Grant 11361072.}}\author {Yongkun Li\\
Department of Mathematics,
Yunnan University\\
Kunming, Yunnan 650091\\
 People's Republic of China}
\date{}
\maketitle{}
\begin{abstract}
In this paper, we first propose two types of concepts of almost periodic functions on the quantum time scale.
Secondly, we study some  basic properties of almost periodic functions on the quantum time scale. Thirdly, based on these, we obtain a result about the existence and uniqueness of almost
periodic solutions of  dynamic equations on the quantum time scale by Lyapunov method. Then, we give a equivalent definition of almost periodic functions on the quantum time scale. Finally, as an application, we propose a class of  high-order Hopfield neural networks on  the  quantum time scale and establish the existence of almost periodic solutions of this class of neural networks.
\end{abstract}
\section{Introduction}

\setcounter{equation}{0}
 \indent \allowdisplaybreaks

 The concept of almost periodicity was  initiated by the
   Bohr  during the period 1923-1925  \cite{b1,b2}.  Bohr's
theory quickly attracted the attention of very famous mathematicians of
that time. And since then the questions of the theory of almost periodic functions and almost periodic solutions
of differential equations have been   very interesting and challenge problems of great importance.
 The interaction between these two theories has enriched
both. On one hand, it is well known that in Celestial Mechanics,
almost periodic solutions and stable solutions are intimately
related. In the same way, stable electronic circuits exhibit
almost periodic behavior.  On the other hand,
certain problems in differential equations have
led to new definitions and results in the theory of almost periodic functions.

In recent years, the theory of  quantum calculus has
received much attention, due to its tremendous applications in several fields of physics, such as cosmic strings and black holes, conformal quantum mechanics, nuclear and high energy physics, fractional quantum Hall effect,  high-$T_c$ superconductors, and so on.  For interested reader, we refer to, for example, [3-6] and the references cited therein.

On one hand, recently, Bohner and Chieochan \cite{q5} introduced in the literature the  concept of periodicity for functions defined on the quantum time scale owing to the fact that   taking into account  the periodicity of  $q$-difference equations is important in order to better understand several physics phenomena.
However, in reality, almost periodic phenomenon is  more common and complicate than periodic one. Therefore, investigating the almost periodicity of  dynamic equations on the quantum time scale
  is more interesting and more challenge.

On the other hand, in order to study the almost periodicity on time scales, a concept of
almost periodic time scales was proposed in \cite{bing5}. Based on this concept, a series of concepts of almost periodic function classes, such as
almost periodic
functions \cite{bing5}, pseudo almost periodic functions \cite{bbbli3},
almost automorphic functions \cite{bbbr1},   weighted pseudo
almost automorphic functions \cite{bbbli2},  almost periodic set-valued functions \cite{bbbset}, almost periodic functions in the sense
of Stepanov on time scales \cite{liwangpan} and so on were defined successively.  Although  the concept of
almost periodic time scales in \cite{bing5} can unify the
continuous and discrete situations effectively, it is very restrictive because it requires the time scale with certain global additivity. This
excludes some interesting time scales. For instance, the quantum time scale,  which has no such global additivity.

Motivated by the above discussion, our main purpose of this paper is  to propose two types of definitions of almost periodic functions on the quantum time scale, to study some  basic properties of almost periodic functions on the quantum time scale, to give an equivalent definition of  almost periodic functions on the quantum time scale and to explore some of their applications to periodic dynamic equations on the quantum time scale.

The organization of this paper is as follows: In Section 2, we
introduce some notations and definitions of time scale calculus. In Section 3, we propose the
concepts of  almost periodic functions on the quantum time
scale and investigate some of their basic properties.  In Section 4,
we study the existence and uniqueness of almost periodic solutions of  dynamic equations on the quantum time
scale by Lyapunov method.   In Section 5, we give an equivalent definition of almost periodic functions on the quantum time
scale. In Section 6, as an application of our results, we first propose a class of  high-order Hopfield neural networks on  the  quantum time scale, then by  the exponential dichotomy of linear dynamic equations on time scales and the Banach fixed point theorem, we establish the existence of almost periodic solutions of this class of neural networks. In Section 7,
 we draw a conclusion.

\section{Preliminaries}
\setcounter{equation}{0}
 \indent

In this section, we shall recall some basic definitions of time scale calculus.

A time scale $\mathbb{T}$ is an arbitrary nonempty closed subset of the real numbers, the forward and backward jump operators $\sigma$, $\rho:\mathbb{T}\rightarrow \mathbb{T}$ and the forward graininess $\mu:\mathbb{T}\rightarrow \mathbb{R}^{+}$ are defined, respectively, by
\[
\sigma(t):=\inf \{s\in\mathbb{T}:s> t\},\,\,\rho(t):=\sup\{s\in\mathbb{T}:s<t\}\,\,
\text{and}\,\,\mu(t)=\sigma(t)-t.
\]

A point $t$ is said to be left-dense if $t>\inf\mathbb{T}$ and $\rho(t)=t$, right-dense if $t<\sup\mathbb{T}$ and $\sigma(t)=t$, left-scattered if $\rho(t)<t$ and right-scattered if $\sigma(t)>t$. If $\mathbb{T}$ has a left-scattered maximum $m$, then $\mathbb{T}^{k}=\mathbb{T}\backslash m$, otherwise $\mathbb{T}^{k}=\mathbb{T}$. If $\mathbb{T}$ has a right-scattered minimum $m$, then $\mathbb{T}_{k}=\mathbb{T}\backslash m$, otherwise $\mathbb{T}^{k}=\mathbb{T}$.

\begin{definition} \cite{bing5}\label{li67}
A time scale is called an almost periodic time scale if
\[
\Pi:=\{\tau\in\mathbb{R}:t+\tau\in\mathbb{T},\forall
t\in\mathbb{T}\}\neq\{0\}.
\]
\end{definition}

A function $f : \mathbb{T}\rightarrow \mathbb{R}$ is right-dense continuous or rd-continuous provided it is continuous at right-dense points in $\mathbb{T}$ and its left-sided limits exist (finite) at left-dense points in $\mathbb{T}$. If $f$ is continuous at each right-dense point and each left-dense point, then $f$ is said to be a continuous function on $\mathbb{T}$.

For $f:\mathbb{T}\rightarrow\mathbb{R}$ and $t\in{\mathbb{T}^{k}}$, then $f$ is called delta differentiable at $t\in{\mathbb{T}}$ if there exists $c\in\mathbb{R}$ such that for given any $\varepsilon\geq{0}$, there is an open neighborhood $U$ of  $t$ satisfying
\[
\left|[f(\sigma(t))-f(s)]-c[\sigma(t)-s]\right|\leq\varepsilon\left|\sigma(t)-s\right|
\]
for all $s\in U$. In this case, $c$ is called the delta derivative of $f$ at $t\in{\mathbb{T}}$, and is denoted by $c=f^{\Delta}(t)$. For $\mathbb{T}=\mathbb{R}$, we have $f^{\Delta}=f^{'}$, the usual derivative, for $\mathbb{T}=\mathbb{Z}$ we have the backward difference operator, $f^{\Delta}(t)=\Delta f(t):=f(t+1)-f(t)$£¬ and for $\mathbb{T}=\overline{q^{\mathbb{Z}}} (q>1)$, the quantum time scale, we have  the $q$-derivative
\[
f^\Delta(t):=D_qf(t)=\displaystyle\left\{\begin{array}{lll}\frac{f(qt)-f(t)}{(q-1)t},&t\neq 0,\\
\lim\limits_{t\rightarrow 0}\frac{fqt)-f(t)}{(q-1)t},&t=0.
\end{array}\right.
\]
\begin{remark}Note that
\[
D_qf(0)=\frac{df(0)}{dt}
\]
if $f$ is continuously differentiable.
\end{remark}
Define for $V\in
C_{rd}[\overline{q^\mathbb{Z}}\times\mathbb{R}^n,\mathbb{R}]$, $D^+_qV^\Delta(t,x(t))$ to
mean that, given $\varepsilon>0$, there exists a right neighborhood
$N_\e\subset N$ such that
\[\frac{1}{\mu(t,s)}[V(\sigma(t),x(\sigma(t)))-V(s,x(\sigma(t))-\mu(t,s)f(t,x(t)))]<D^+_qV^\Delta(t,x(t))+\e\]
for each $s\in N_\e$, $s>t$, where $\mu(t,s)\equiv\sigma(t)-s$. If
$t$ is right-scattered and $V(t,x(t))$ is continuous at $t$, this
reduces to
\[D^+_qV^\Delta(t,x(t))=\frac{V(\sigma(t),x(\sigma(t)))-V(t,x(t))}{\sigma(t)-t}.\]
\begin{lemma}\cite{t1}\label{lmas}
Let $y,f\in C_{rd}$ and $p\in\mathcal{R}^+$, then
\[D_qy (t)\leq p(t)y(t)+f(t)\,\,\,\,\,{\rm for\,\,all}\,\,\,\,t\in \overline{q^\mathbb{Z}}\]
implies
\[y(t)\leq y(t_0)e_p(t,t_0)+\int^t_{t_0}e_p(t,\sigma(\tau))f(\tau)d_q\tau\,\,\,\,{\rm for\,\,all}\,\,\,\,t\in \overline{q^\mathbb{Z}},\]
where $t_0\in \overline{q^\mathbb{Z}}$.
\end{lemma}
For more details about the theory of time scale calculus and the theory of quantum calculus, the reader may want to consult \cite{q1, t1,t2}.

\section{Almost periodic functions }
\setcounter{equation}{0}\indent

From now on, we use $\mathbb{X}$ to denote the complex numbers $\mathbb{C}$ or the real numbers $\mathbb{R}$, use $D$ to denote an open set in $\mathbb{X}$ or
$D=\mathbb{X}$, and use $S$ to denote an arbitrary compact subset of $D$.

\begin{definition}\label{def34}
 A function  $f\in
C(\overline{q^{\mathbb{Z}}}\times D,\mathbb{X})$ is  called an almost
periodic function in $t\in \overline{q^{\mathbb{Z}}}$ uniformly in $x\in D$ if   for any given $\varepsilon>0$
and each compact subset $S$ of $D$, there exists
$\ell(\varepsilon,S)>0$ such that each interval
$[q^p, q^pq^\ell]\cap{\overline{q^{\mathbb{Z}}}}$ contains   a $q^\tau(p,\varepsilon,S)\in [q^p, q^pq^\ell]\cap{\overline{q^{\mathbb{Z}}}} (p\in \mathbb{Z})$ such that
\begin{equation*}
|f(t q^\tau ,x)-f(t,x)|<\varepsilon, \quad\forall (t,x)\in
\overline{q^{\mathbb{Z}}}\times S.
\end{equation*}
This   $q^\tau$ is called the $\varepsilon$-translation number of $f$ and $\ell(\varepsilon,S)$ is called  the
inclusion length of the set $E\{\varepsilon,f,S\}=\{q^\tau\in \overline{q^{\mathbb{Z}}}:|f(t q^\tau ,x)-f(t,x)|<\varepsilon,\,\,
\forall (t,x)\in   \overline{q^{\mathbb{Z}}}\times S\}.$
\end{definition}

\begin{definition}\label{def34b}
 A function  $f\in
C(\overline{q^{\mathbb{Z}}}\times D,\mathbb{X})$ is  called an almost
periodic function in $t\in \overline{q^{\mathbb{Z}}}$ uniformly in $x\in D$ if  for any given $\varepsilon>0$
and each compact subset $S$ of $D$, there exists
$\ell(\varepsilon,S)>0$ such that each interval
$[q^p, q^pq^\ell]\cap{\overline{q^{\mathbb{Z}}}}$ contains  a $q^\tau(p,\varepsilon,S)\in [q^p, q^pq^\ell]\cap{\overline{q^{\mathbb{Z}}}} (p\in \mathbb{Z})$ such that
\begin{equation*}
|q^\tau f(t q^\tau,x)-f(t,x)|<\varepsilon, \quad\forall (t,x)\in
(\overline{q^{\mathbb{Z}}}\setminus\{0\})\times S.
\end{equation*}
This   $q^\tau$ is called the $\varepsilon$-translation number of $f$ and $\ell(\varepsilon,S)$ is called  the
inclusion length of the set $E\{\varepsilon,f,S\}=\{q^\tau\in \overline{q^{\mathbb{Z}}}:|q^\tau f(t q^\tau ,x)-f(t,x)|<\varepsilon,\,\,
\forall (t,x)\in   (\overline{q^{\mathbb{Z}}}\setminus\{0\})\times S\}.$
\end{definition}

\begin{definition}\label{def34y}
 A function  $f\in
C(\overline{q^{\mathbb{Z}}},\mathbb{X})$ is  called an almost
periodic function  if   for any given $\varepsilon>0$, there exists
$\ell(\varepsilon)>0$ such that each interval
$[q^p, q^pq^\ell]\cap{\overline{q^{\mathbb{Z}}}}$ contains  a $q^\tau(p,\varepsilon)\in [q^p, q^pq^\ell]\cap{\overline{q^{\mathbb{Z}}}} (p\in \mathbb{Z})$ such that
\begin{equation*}
|f(t q^\tau)-f(t)|<\varepsilon, \quad\forall t\in
\overline{q^{\mathbb{Z}}}.
\end{equation*}
This   $q^\tau$ is called the $\varepsilon$-translation number of $f$ and $\ell(\varepsilon)$ is called  the
inclusion length of the set $E\{\varepsilon,f\}=\{q^\tau\in \overline{q^{\mathbb{Z}}}:|f(t q^\tau)-f(t)|<\varepsilon,\,\,
\forall t\in   \overline{q^{\mathbb{Z}}}\}.$
\end{definition}

\begin{definition}\label{def34by}
 A function  $f\in
C(\overline{q^{\mathbb{Z}}},\mathbb{X})$ is  called an almost
periodic function  if   for any given $\varepsilon>0$, there exists
$\ell(\varepsilon)>0$ such that each interval
$[q^p, q^pq^\ell]\cap{\overline{q^{\mathbb{Z}}}}$ contains  a $q^\tau(p,\varepsilon)\in [q^p, q^pq^\ell]\cap{\overline{q^{\mathbb{Z}}}} (p\in \mathbb{Z})$ such that
\begin{equation*}
|q^\tau f(t q^\tau)-f(t)|<\varepsilon, \quad\forall t\in
\overline{q^{\mathbb{Z}}}\setminus\{0\}.
\end{equation*}
This  $q^\tau$ is called the $\varepsilon$-translation number of $f$ and $\ell(\varepsilon)$ is called  the
inclusion length of the set $E\{\varepsilon,f\}=\{q^\tau\in \overline{q^{\mathbb{Z}}}:|q^\tau f(t q^\tau)-f(t)|<\varepsilon,\,\,
\forall  t\in   \overline{q^{\mathbb{Z}}}\}.$
\end{definition}

\begin{definition}\label{zht1}
 A function $f\in C(\overline{q^{\mathbb{Z}}},
\mathbb{E}^n)$ is called an asymptotically almost periodic function
 if
\[f(t)=p(t)+q(t),\]
where $p$ is an almost periodic function on $\overline{q^\mathbb{Z}}$, and
$q(t)\rightarrow 0$ as $t\rightarrow\infty$.
\end{definition}

For convenience, we introduce some notations:

 Let  $BC(\overline{q^{\mathbb{Z}}}\times D,\mathbb{X})=\{ f\in
C(\overline{q^{\mathbb{Z}}}\times D,\mathbb{X})$ is  bounded on  $\overline{q^{\mathbb{Z}}}\times D$$\}$, $AP(\overline{q^{\mathbb{Z}}}\times D,\mathbb{X})=\{ f\in
C(\overline{q^{\mathbb{Z}}}\times D,\mathbb{X})$ is   almost
periodic   in $t\in \overline{q^{\mathbb{Z}}}$ uniformly in $x\in D$$\}$ and $AP(\overline{q^{\mathbb{Z}}},\mathbb{X})=\{ f\in
C(\overline{q^{\mathbb{Z}}},\mathbb{X})$ is   almost
periodic$\}$.

 Let
$\alpha=\{\alpha_{n}\}\subset \mathbb{Z}$ and $\beta=\{\beta_{n}\}\subset \mathbb{Z}$ be two sequences of integer numbers. We denote
$\alpha+\beta:=\{\alpha_{n}+\beta_{n}\},
-\alpha:=\{-\alpha_{n}\}$ and use $\beta\subset\alpha$ to denote that $\beta$ is a subsequence of
$\alpha$.
We say $\alpha$ and $\beta$ are common
subsequences of $\alpha^{'}$ and $\beta^{'}$, respectively, if there is some given function $n(\nu)$ such that
 $\alpha_{n}=\alpha^{'}_{n(\nu)}$ and $\beta_{n}=\beta^{'}_{n(\nu)}$.

We  introduce two translation operators $T$ and $\hat{T}$,
$T_{\alpha}f(t,x)=g(t,x)$ and $\hat{T} f(t,x)=g(t,x)$ mean that
$g(t,x)=\lim\limits_{n\rightarrow+\infty}f(q^{\alpha_{n}}t,x)$ and $g(t,x)=\lim\limits_{n\rightarrow+\infty}q^{\alpha_{n}}f(q^{\alpha_{n}}t,x) (t\neq 0)$, respectively, and are
written only when the limits exist. The mode of convergence  will be specified at each use of the
symbols.

\begin{remark}
When study the properties of almost periodic functions defined by  Definitions \ref{def34} and \ref{def34y}, we use the translation operator $T$.
When study the properties of almost periodic functions defined by  Definitions \ref{def34b} and \ref{def34by}, we use the translation operator $\bar{T}$.
\end{remark}

\begin{remark}
All the results of this section hold for almost periodic functions defined by Definitions \ref{def34}, \ref{def34b}, \ref{def34y} and \ref{def34by}. Since their proofs are similar, we only prove them under Definitions \ref{def34} and \ref{def34y}.
\end{remark}

\begin{remark}
 A function $f:\overline{q^{\mathbb{Z}}}\times D\rightarrow\mathbb{X}$ is continuous if and only if $\lim\limits_{n\rightarrow-\infty}f(q^n,x)=f(0,x)$ exists uniformly in $x\in S$.
\end{remark}

\begin{theorem}\label{thm33}
Let   $f\in AP(\overline{q^{\mathbb{Z}}}\times D,\mathbb{X})$,  then it is uniformly continuous and
bounded on $\overline{q^{\mathbb{Z}}}\times S$.
\end{theorem}
\begin{proof}
For a given $\varepsilon_0\leq 1$ and some compact subset $S\subset D$,
there exists  an $\ell(\varepsilon,S)$ such that in each interval
 $[q^p,q^pq^\ell]\cap\overline{q^{\mathbb{Z}}} (p\in \mathbb{Z})$, there exists a $q^\tau\in
E\{\varepsilon_0,f,S\}$ such that
\begin{equation*}
|f(q^\tau t,x)-f(t,x)|<\varepsilon_0\leq1, \quad\forall
(t,x)\in \overline{q^{\mathbb{Z}}}\times S.
\end{equation*}
It follows from $f\in C(\overline{q^{\mathbb{Z}}}\times D,\mathbb{X})$ that for any $(t,x)\in
([1,q^\ell]\cap\overline{q^{\mathbb{Z}}})\times S$, there exists an
 $M>0$ such that $|f(t,x)|<M$. For any given $t\in\overline{q^{\mathbb{Z}}}$ and $t\neq 0$,
 take $q^\tau\in
E(\varepsilon,f,S)\cap[t^{-1},t^{-1}q^\ell]$, then
$tq^\tau\in[1,q^\ell]\cap{\overline{q^{\mathbb{Z}}}}$. Hence, for $x\in
S$, we get
\begin{equation*}
|f(tq^\tau,x)|<M
\quad
\mathrm{and}
\quad
|f(tq^\tau,x)-f(t,x)|<1.
\end{equation*}
Therefore, we obtain
\begin{equation*}
|f(t,x)|<\max\big\{\max\limits_{x\in S}|f(0,x)|,M\big\}+1,\,\, \forall (t,x)\in \overline{q^{\mathbb{Z}}}\times S.
\end{equation*}

Furthermore, because of the continuity of  $f$  at $(0,x), x\in S$,
for any given $\varepsilon>0$, there exists a positive integer $N_0(\varepsilon)$ such that $|f(q^n,x)-f(0,x)|<\frac{\varepsilon}{2}$ for all $n<-N_0$ and $x\in S$.
Thus,  for any  $(q^{n_1},x),(q^{n_2},x)\in ([0, q^{-N_0}]\cap \overline{q^{\mathbb{Z}}})\times S$, we find
 \begin{equation}
 |f(q^{n_1},x)-f(q^{n_2},x)|<|f(q^{n_1},x)-f(0,x)|+|f(q^{n_2},x)-f(0,x)|<\varepsilon,
 \end{equation}
 therefore, $f(t,x)$ is uniformly continuous on $([0, q^{-N_0}]\cap \overline{q^{\mathbb{Z}}})\times S$.
 Since all the points in the interval $[q^{-N_0}, +\infty)\cap \overline{q^{\mathbb{Z}}}$ are isolated points and $\mu (t)\geq (q-1)q^{-N_0}$ for all $t\in [q^{-N_0}, +\infty)\cap \overline{q^{\mathbb{Z}}}$, so $f(t,x)$ is uniformly continuous on $([q^{-N_0}, +\infty)\cap \overline{q^{\mathbb{Z}}})\times S$. The proof is completed.
\end{proof}

\begin{theorem}\label{thm34}
Let $f\in AP(\overline{q^{\mathbb{Z}}}\times D,\mathbb{X})$, then for any given sequence
$\alpha^{'}\subset\mathbb{Z},$ there exist a subsequence
$\beta\subset\alpha^{'}$ and $g\in C(\overline{q^{\mathbb{Z}}}\times
D,\mathbb{X})$ such that $T_{\beta}f=g$ holds
uniformly on $\overline{q^{\mathbb{Z}}}\times S$ and $g \in AP(\overline{q^{\mathbb{Z}}}\times D,\mathbb{X})$.
\end{theorem}
\begin{proof}
 For any given $\varepsilon>0$ and sequence
$\alpha^{'}=\{\alpha_{n}^{'}\}\subset \mathbb{Z},$ we denote
$\alpha_{n}^{'}=\tau_{n}^{'}+\gamma_{n}^{'},$ where $q^{\tau_{n}^{'}}\in
E\{\DF{\varepsilon}{2},f,S\}, \gamma_{n}^{'}\in \mathbb{Z}$ and $0\leq\gamma_{n}^{'}\leq
\ell(\frac{\varepsilon}{2},S),n=1,2,\ldots.$\big. Hence, it is easy to see that
there exists a subsequence
$\gamma=\{\gamma_{n}\}\subset\gamma^{'}=\{\gamma_{n}^{'}\}$ such
that $\gamma_{n}=s$ for $n\in \mathbb{N}$, $0\leq s\leq
\ell.$
We can take
$\alpha\subset\alpha^{'},\,\tau\subset\tau^{'}$
such that $\alpha,\,\tau$ common with $\gamma$  and $\alpha_n=\tau_n+s$ for all $n\in \mathbb{N}$, then
\begin{equation*}
|f(tq^{\tau_{p}-\tau_{m}},x)-f(t,x)|\leq|f(tq^{\tau_{p}-\tau_{m}},x)-f(tq^{\tau_{p}},x)|+|f(tq^{\tau_{p}},x)-f(t,x)|<\varepsilon,
\end{equation*}
which implies that
\begin{equation*}
q^{\alpha_{p}-\alpha_{m}}=q^{\tau_{p}-\tau_{m}}\in
E\{\varepsilon,f,S\}.
\end{equation*}
Hence, we can get
{\setlength\arraycolsep{2pt}\begin{eqnarray*}
|f(tq^{\alpha_{p}},x)-f(tq^{\alpha_{m}},x)|&\leq&\sup\limits_{(t,x)\in\overline{q^{\mathbb{Z}}}\times S}|f(tq^{\alpha_{p}},x)-f(tq^{\alpha_{m}},x)|\\
&\leq&\sup\limits_{(t,x)\in\overline{q^{\mathbb{Z}}}\times S}|f(tq^{\alpha_{p}-\alpha_{m}},x)-f(t,x)|<\varepsilon.
\end{eqnarray*}}
Thus, we can take subsequences $\alpha^{(k)}=\{\alpha_{n}^{(k)}\},$
$k=1,2,\ldots,$ and $\alpha^{(k+1)}\subset\alpha^{(k)}\subset\alpha$
such that for any integers $m,p,$ and all $(t,x)\in\overline{q^{\mathbb{Z}}}\times
S$,
\begin{equation*}
|f(tq^{\alpha_{p}^{(k)}},x)-f(tq^{\alpha_{m}^{(k)}},x)|<\frac{1}{k},\,k=1,2,\ldots.
\end{equation*}
For all sequences $\alpha^{(k)},k=1,2,\ldots$, we can take a
sequence $\beta=\{\beta_{n}\},\,\beta_{n}=\alpha_{n}^{(n)},$ then it
is easy to see that for any integers $p,m$ with $p<m$ and all
$(t,x)\in\overline{q^{\mathbb{Z}}}\times S$,
\begin{equation*}
|f(tq^{\beta_{p}},x)-f(tq^{\beta_{m}},x)|<\frac{1}{p}.
\end{equation*}
Therefore, $\{f(tq^{\beta_{n}},x)\}$ converges uniformly on
$\overline{q^{\mathbb{Z}}}\times S$, that is,  $T_{\beta}f=g$ holds
uniformly on $\overline{q^{\mathbb{Z}}}\times S$.

Next, we  show that $g\in C(\overline{q^{\mathbb{Z}}}\times D,\mathbb{X})$. Otherwise, there exists a point
$(t_0,x_0)\in\overline{q^{\mathbb{Z}}}\times D$ such that $g$ is not
continuous at this point. Then there  exist $\varepsilon_{0}>0$
and sequences $\{\delta_{m}\}\subset \mathbb{R}^+,\,\{t_{m}\}\subset \overline{q^\mathbb{Z}},\,\{x_{m} \}\subset D,$ where
$\delta_{m}>0,$   $\delta_{m}\rightarrow0$ as
$m\rightarrow+\infty$, $|t_{0}-t_{m}|+|x_{0}-x_{m}|<\delta_{m}$
and
\begin{equation}\label{e31}
|g(t_{0},x_{0})-g(t_{m},x_{m})|\geq\varepsilon_{0}.
\end{equation}
Let $\mathcal{X}=\{x_{m}\}\bigcup\{x_0\},$ obviously, $\mathcal{X}$ is a compact subset
of $D$. Since $T_{\beta}f=g$ holds
uniformly on $\overline{q^{\mathbb{Z}}}\times S$,  there exists a positive integer
$N=N(\varepsilon_{0},\mathcal{X})$ such that for $n>N$,
\begin{equation}\label{e32}
|f(t_{m}q^{\beta_{n}},x_{m})-g(t_{m},x_{m})|<\frac{\varepsilon_0}{3}\quad\forall m\in\mathbb{Z}^{+}
\end{equation}
and
\begin{equation}\label{e33}
|f(t_{0}q^{\beta_{n}},x_{0})-g(t_{0},x_{0})|<\frac{\varepsilon_0}{3}.
\end{equation}
According to the uniform continuity of $f$ on $\overline{q^{\mathbb{Z}}}\times
\mathcal{X}$, for sufficiently large $m$, we have
\begin{equation}\label{e34}
|f(t_{0}q^{\beta_{n}},x_{0})-f(t_{m}q^{\beta_{n}},x_{m})|<\frac{\varepsilon_0}{3}.
\end{equation}
From \eqref{e32}-\eqref{e34}, we get
\begin{equation*}
|g(t_{0},x_{0})-g(t_{m},x_{m})|<\varepsilon_{0},
\end{equation*}
which   contradicts \eqref{e31}. Therefore, $g$ is
continuous on $\overline{q^{\mathbb{Z}}}\times D$.

Finally, for any compact subset $S\subset D$ and given $\varepsilon>0$,
we take $q^\tau\in E\{\varepsilon,f,S\},$ then
\begin{equation*}
|f(tq^{\beta_{n}+\tau},x)-f(tq^{\beta_{n}},x)|<\varepsilon,\,\, \forall (t,x)\in\overline{q^{\mathbb{Z}}}\times S.
\end{equation*}
Let $n\rightarrow+\infty,$   we
have
\begin{equation*}
|g(tq^\tau,x)-g(t,x)|\leq\varepsilon,\,\, \forall (t,x)\in\overline{q^{\mathbb{Z}}}\times S,
\end{equation*}
which implies that
$g \in AP(\overline{q^{\mathbb{Z}}}\times D,\mathbb{X})$. This
completes the proof.
\end{proof}

\begin{theorem}\label{thm35}
Let $f\in BC(\overline{q^{\mathbb{Z}}}\times D,\mathbb{X}),$ if for any
sequence $\alpha^{'}\subset \mathbb{Z}$, there exists
a subsequence $\alpha\subset\alpha^{'}$ such that $T_{\alpha}f$ exists
uniformly on $\overline{q^{\mathbb{Z}}}\times S$, then $f \in AP(\overline{q^{\mathbb{Z}}}\times D,\mathbb{X})$.
\end{theorem}
\begin{proof}
Proof by contradiction. If $f \not\in AP(\overline{q^{\mathbb{Z}}}\times D,\mathbb{X})$, then there exist
$\varepsilon_{0}>0$ and $S\subset D$ such that for no matter how
large $\ell>0$, we can always find an interval  $[q^p, q^pq^\ell]\cap \overline{q^{\mathbb{Z}}}$ such that
$[q^p, q^pq^\ell]\cap E\{\varepsilon_{0},f,S\}=\emptyset.$

To this end, we take a number $\alpha_{1}^{'}\in \mathbb{Z}$ and find an interval
$[q^{a_{1}},q^{b_{1}}]\cap \overline{q^{\mathbb{Z}}}$ with $b_{1}-a_{1}>2|\alpha_{1}^{'}|$ and $a_1,b_1\in 2\mathbb{Z}$ such that
$[q^{a_1}, q^{b_1}]\cap E\{\varepsilon_{0},f,S\}=\emptyset.$
 Next, we take $\alpha_{2}^{'}=\frac{1}{2}(a_1+b_1),$
it is easy to see that $q^{\alpha_{2}^{'}-\alpha_{1}^{'}}\in[q^{a_1},q^{b_1}]\cap{\overline{q^{\mathbb{Z}}}}$, and so
$q^{\alpha_{2}^{'}-\alpha_{1}^{'}}\not\in E\{\varepsilon_{0},f,S\};$
then we find an interval $[q^{a_{2}},q^{b_{2}}]\cap \overline{q^{\mathbb{Z}}}$ with
$b_{2}-a_{2}>2(|\alpha_{1}^{'}|+|\alpha_{2}^{'}|)$ and $a_2,b_2\in 2\mathbb{Z}$ such that $[q^{a_2}, q^{b_2}]\cap E\{\varepsilon_{0},f,S\}=\emptyset.$
 Next, we take $\alpha_{3}^{'}=\frac{1}{2}(a_2+b_2),$
obviously,
$q^{\alpha_{3}^{'}-\alpha_{2}^{'}},\,q^{\alpha_{3}^{'}-\alpha_{1}^{'}}\not\in
E\{\varepsilon_{0},f,S\}.$ We can repeat these process again and
again, and can find $\alpha_{4}^{'},\,\alpha_{5}^{'},\,\ldots,$ such
that $q^{\alpha_{i}^{'}-\alpha_{j}^{'}}\not\in
E\{\varepsilon_{0},f,S\},\,i>j.$ Thus, for any $i\neq
j,\,i,j=1,2,\ldots,$ without loss of generality, let $i>j,$   we find
\begin{eqnarray*}
\sup\limits_{(t,x)\in\overline{q^{\mathbb{Z}}}\times
S}|f(tq^{\alpha_{i}^{'}},x)-f(tq^{\alpha_{j}^{'}},x)|&=&\sup\limits_{(t,x)\in\overline{q^{\mathbb{Z}}}\times
S}|f(tq^{\alpha_{i}^{'}-\alpha_{j}^{'}},x)-f(t,x)|\geq\varepsilon_{0},
\end{eqnarray*}
which implies that there is no uniformly convergent subsequence of
$\{f(tq^{\alpha_{n}^{'}},x)\}$ for $(t,x)\in \overline{q^{\mathbb{Z}}}\times S$. This
is a contradiction. Hence, $f \in AP(\overline{q^{\mathbb{Z}}}\times D,\mathbb{X})$. This completes the proof.
\end{proof}

\begin{definition}
\label{def35}
A function $f \in C(\overline{q^{\mathbb{Z}}}\times D,\mathbb{X})$ is called an almost
periodic function in $t$ uniformly in $x\in D$, if for any
given sequence $\alpha^{'}\subset \mathbb{Z},$ there exists a subsequence
$\alpha\subset\alpha^{'}$ such that $T_{\alpha}f$ exists
uniformly on $\overline{q^{\mathbb{Z}}}\times S.$
\end{definition}
\begin{definition}
\label{def35a}
A function $f \in C(\overline{q^{\mathbb{Z}}}\times D,\mathbb{X})$ is called an almost
periodic function in $t$ uniformly in $x\in D$, if for any
given sequence $\alpha^{'}\subset \mathbb{Z},$ there exists a subsequence
$\alpha\subset\alpha^{'}$ such that $\hat{T}_{\alpha}f$ exists
uniformly on $\overline{q^{\mathbb{Z}}}\times S.$
\end{definition}

According to Theorems \ref{thm34} and   \ref{thm35}, one can easily see that  Definitions \ref{def35} and \ref{def35a} are equivalent to Definitions \ref{def34} and \ref{def34b}, respectively.

\begin{theorem}\label{thm36}
If $f \in AP(\overline{q^{\mathbb{Z}}}\times D,\mathbb{X})$  and $\varphi\in AP(\overline{q^{\mathbb{Z}}},\mathbb{X})$ with $\varphi(t)\subset S$ for all $t\in\overline{q^{\mathbb{Z}}}$. Then
$f(\cdot,\varphi(\cdot))\in AP(\overline{q^{\mathbb{Z}}},\mathbb{X})$.
\end{theorem}

\begin{proof}It is easy to see that
for any  sequence $\alpha^{'}\subset\mathbb{Z},$ there exist a subsequence
$\alpha\subset\alpha^{'}$, functions $\psi$ and $g$ such that
$T_{\alpha}\varphi=\psi$ exists uniformly on $\overline{q^{\mathbb{Z}}}$ and
$T_{\alpha}f=g$ exists uniformly on $\overline{q^{\mathbb{Z}}}\times S,$
where $\psi\in AP(\overline{q^{\mathbb{Z}}},\mathbb{X})$  and $g\in AP(\overline{q^{\mathbb{Z}}}\times D,\mathbb{X})$. Hence, $g$ is uniformly
continuous on $\overline{q^{\mathbb{Z}}}\times S$, then for any given
$\varepsilon>0$, there exists $\delta(\frac{\varepsilon}{2})>0$ such
that for any $x_{1},x_{2}\in S$ and
$|x_{1}-x_{2}|<\delta(\frac{\varepsilon}{2})$, we have
\begin{equation*}
|g(t,x_1)-g(t,x_2)|<\frac{\varepsilon}{2}, \,\, \forall t\in\overline{q^{\mathbb{Z}}}.
\end{equation*}
In addition, there exists an   $N_{0}(\varepsilon)>0$ such that for $n\geq
N_{0}(\varepsilon)$,
\begin{equation*}
|f(tq^{\alpha_{n}},x)-g(t,x)|<\frac{\varepsilon}{2},\quad
\forall(t,x)\in\overline{q^{\mathbb{Z}}}\times S
\end{equation*}
and
\begin{equation*}
|\varphi(tq^{\alpha_{n}})-\psi(t)|<\delta(\frac{\varepsilon}{2}),\quad
\forall t\in\overline{q^{\mathbb{Z}}},
\end{equation*}
where $\varphi(tq^\alpha_{n})\subset S, \psi(t)\subset S$ for all
$t\in\overline{q^{\mathbb{Z}}}.$ Therefore, for $n\geq N_{0}(\varepsilon)$,
\begin{eqnarray*}
|f(tq^{\alpha_{n}},\varphi(tq^{\alpha_{n}}))-g(t,\psi(t))|&\leq&|f(tq^{\alpha_{n}},\varphi(tq^{\alpha_{n}}))-g(t,\varphi(tq^{\alpha_{n}}))|\\
&&+|g(t,\varphi(tq^{\alpha_{n}}))-g(t,\psi(t))|<\varepsilon,
\end{eqnarray*}
which implies that $T_{\alpha}f(t,\varphi(t))=g(t,\psi(t))$ exists uniformly
on $\overline{q^{\mathbb{Z}}}$. Thus, $f(\cdot,\varphi(\cdot))\in AP(\overline{q^{\mathbb{Z}}},\mathbb{X})$. The proof is complete.
\end{proof}

\begin{definition}\label{def36}
 Let $f\in BC(\overline{q^{\mathbb{Z}}}\times D,\mathbb{X})$, then
$H(f)=\{g:\overline{q^{\mathbb{Z}}}\times D\rightarrow\mathbb{X}|$ there exits
$\alpha\subset\mathbb{Z}$ such that $T_{\alpha}f=g$ exists uniformly
on $\overline{q^{\mathbb{Z}}}\times S$$\}$ is called the hull of $f$.
\end{definition}

\begin{theorem}\label{thm37}
$H(f)$ is compact in uniform norm if and only if $f\in AP(\overline{q^{\mathbb{Z}}}\times D,\mathbb{X})$.
\end{theorem}
\begin{proof}
If $H(f)$ is a compact set, then for each sequence $\alpha^{'}\subset\mathbb{Z}$,  there  exists a subsequence
$\alpha\subset \alpha'$ such that
$T_{\alpha}f$ exists uniformly on $\overline{q^{\mathbb{Z}}}\times D$.

Conversely, if $f \in AP(\overline{q^{\mathbb{Z}}}\times D,\mathbb{X})$ and $\{g_{n}\}\subset H(f)$, then we can choose
$\alpha^{'}=\{\alpha^{'}_{n}\}$ such that
\begin{equation*}
|f(tq^{\alpha^{'}_{n}},x)-g_{n}(t,x)|<\frac{1}{n},\quad\forall
(t,x)\in\overline{q^{\mathbb{Z}}}\times S.
\end{equation*}
So, we can find a subsequence $\alpha\subset \alpha^{'}$ such that $T_{\alpha}f$
exists uniformly  on $\overline{q^{\mathbb{Z}}}\times D$. Let $\beta\subset\gamma=\{n\}$ such that $\beta$ and
$\alpha$ are common subsequences, then
\begin{equation*}
f(tq^{\alpha_{n}},x)-g_{\beta_{n}}(t,x)\rightarrow 0\quad
(n\rightarrow\infty),\quad\forall (t,x)\in\overline{q^{\mathbb{Z}}}\times S,
\end{equation*}
hence
\begin{equation*}
\lim\limits_{n\rightarrow\infty}g_{\beta_n}(t,x)=T_{\alpha}f(t,x)\in
H(f),\quad \forall (t,x)\in\overline{q^{\mathbb{Z}}}\times S,
\end{equation*}
therefore, $H(f)$ is a compact set. The proof is complete.
\end{proof}

\begin{theorem}\label{thm38}
If $f \in AP(\overline{q^{\mathbb{Z}}}\times D,\mathbb{X})$, then
for any $g\in H(f), H(f)=H(g)$.
\end{theorem}
\begin{proof}
For any $h\in H(g)$ there exists $\alpha^{'}\subset \mathbb{Z}$ such
that $T_{\alpha^{'}} g=h$ exists uniformly
on $\overline{q^{\mathbb{Z}}}\times S$. Since $f \in AP(\overline{q^{\mathbb{Z}}}\times D,\mathbb{X})$,  one can extract a subsequence
$\alpha\subset \alpha'$ such that $T_\alpha
f(t,x)=\lim\limits_{n\rightarrow\infty}f(tq^{\alpha_{n}},x)$ exists
uniformly on $\overline{q^{\mathbb{Z}}}\times S$.
Because $g\in H(f)$, so there exists $\alpha^{(1)}\subset \mathbb{Z}$
such that
\begin{equation*}
\lim\limits_{n\rightarrow+\infty}f(tq^{\alpha_{n}^{(1)}},x)=g(t,x),\quad
\forall (t,x)\in\overline{q^{\mathbb{Z}}}\times S,
\end{equation*}
hence,
\begin{equation*}
\lim\limits_{n\rightarrow+\infty}f(tq^{\alpha_{n}^{(1)}+\alpha_n},x)=g(tq^{\alpha_n},x),\quad\forall
(t,x)\in\overline{q^{\mathbb{Z}}}\times S,
\end{equation*}
then we can take $\beta =\{\beta_n\}=\{\alpha_{n}^{(1)}+\alpha_n\}$
such that
\begin{equation*}
|f(tq^{\beta_{n}},x)-g(tq^{\alpha_{n}},x)|<\frac{1}{n}, \quad \forall
(t,x)\in\overline{q^{\mathbb{Z}}}\times S.
\end{equation*}
It follows from $T_{\beta}f=T_\alpha
g=T_{\alpha^{'}}g=h$  that $h\in H(f)$. Thus,
$H(g)\subseteq H(f)$.

On the other hand, for any $g \in H(f)$, there exists
$\alpha$ such that $T_{\alpha}f=g$ exists uniformly
on $\overline{q^{\mathbb{Z}}}\times S$, hence
\begin{equation*}
|f(tq^{\alpha_{n}},x)-g(t,x)|\rightarrow 0,\quad
n\rightarrow\infty,\,\,\forall (t,x)\in\overline{q^{\mathbb{Z}}}\times S.
\end{equation*}
Let $tq^{\alpha_{n}}=s$, we obtain
\begin{equation*}
|f(s,x)-g(sq^{-\alpha_{n}},x)|\rightarrow 0,\quad
n\rightarrow\infty,\,\,\forall
(s,x)\in\overline{q^{\mathbb{Z}}}\times S,
\end{equation*}
that is, $T_{-\alpha}g=f$ exists uniformly
on $\overline{q^{\mathbb{Z}}}\times S$. Thus, $f\in H(g)$.
Therefore, $H(f)\subseteq H(g)$.
The proof is complete.
\end{proof}

From  Definition \ref{def36} and Theorem \ref{thm38}, one can easily show that
\begin{theorem}\label{thm39}
If $f \in AP(\overline{q^{\mathbb{Z}}}\times D,\mathbb{X})$, then
for any $g\in H(f)$, $g \in AP(\overline{q^{\mathbb{Z}}}\times D,\mathbb{X})$.
\end{theorem}

\begin{theorem}\label{thm313}
If $f,g \in AP(\overline{q^{\mathbb{Z}}}\times D,\mathbb{X})$, then
\begin{itemize}
  \item [$(a)$]$f+g,fg \in AP(\overline{q^{\mathbb{Z}}}\times D,\mathbb{X})$,
  \item [$(b)$] if $\inf\limits_{t\in\overline{q^{\mathbb{Z}}}}|g(t,x)|\geq M>0$, then
$\frac{f}{g}\in AP(\overline{q^{\mathbb{Z}}}\times D,\mathbb{X})$.
\end{itemize}
\end{theorem}
\begin{proof}

$(a)$  Since $f,g \in AP(\overline{q^{\mathbb{Z}}}\times D,\mathbb{X})$,  from any sequence
$\{\alpha_{n}^{'}\}\subset \mathbb{Z}$ one can extract a subsequence
$\{\alpha_{n}\}$ such that $T_\alpha
f=\lim\limits_{n\rightarrow\infty}f(tq^{\alpha_{n}},x)$ exists
uniformly on $\overline{q^{\mathbb{Z}}}\times S$. Then from the $\{\alpha_{n}\}\subset \mathbb{Z}$ we can extract a subsequence
$\{\beta_{n}'\}$ such that $T_{\beta'}
g=\lim\limits_{n\rightarrow\infty}g(tq^{\beta_{n}'},x)$ exists
uniformly on $\overline{q^{\mathbb{Z}}}\times S$. Consequently, we can extract a subsequence
$\{\beta_{n} \}\subset \{\beta_n'\}$ such that $T_\beta
(f+g)=\lim\limits_{n\rightarrow\infty}(f(tq^{\beta_{n}},x)+g(tq^{\beta_{n}},x))$ and $T_\beta
(fg)=\lim\limits_{n\rightarrow\infty}(f(tq^{\beta_{n}},x)g(tq^{\beta_{n}},x))$ exists
uniformly on $\overline{q^{\mathbb{Z}}}\times S$. That is, $f+g,fg\in AP(\overline{q^{\mathbb{Z}}}\times D,\mathbb{X})$.

$(b)$ Because $g \in AP(\overline{q^{\mathbb{Z}}}\times D,\mathbb{X})$, so for any given $\varepsilon>0$
and each compact subset $S$ of $D$, there exists
$\ell(\varepsilon,S)>0$ such that each interval
$[q^p, q^pq^\ell]\cap{\overline{q^{\mathbb{Z}}}}$ contains   a $q^\tau(p,\varepsilon,S)\in [q^p, q^pq^\ell]\cap{\overline{q^{\mathbb{Z}}}} (p\in \mathbb{Z})$ such that
\begin{equation*}
|g(t q^\tau ,x)-g(t,x)|<M^2\varepsilon, \quad\forall t\in
\overline{q^{\mathbb{Z}}}\times S.
\end{equation*}
Now
\begin{eqnarray*}
\bigg|\frac{1}{g(tq^\tau,x)}-\frac{1}{g(t,x)}\bigg|&=&\frac{|g(tq^\tau,x)-g(t,x)|}{|g(tq^\tau,x)g(t,x)|}\\
&\leq&\frac{|g(tq^\tau,x)-g(t,x)|}{M^2}<\varepsilon
\end{eqnarray*}
for all $(t,x)\in \overline{q^{\mathbb{Z}}}\times S$. That is, $\frac{1}{g}\in AP(\overline{q^{\mathbb{Z}}}\times D,\mathbb{X})$. Now using   From $(a)$ it follows that  $f/g \in AP(\overline{q^{\mathbb{Z}}}\times D,\mathbb{X})$.
 The proof is
complete.
\end{proof}

\begin{theorem}
If $f_{n}\in AP(\overline{q^{\mathbb{Z}}}\times D,\mathbb{X}),n=1,2,\ldots$ and the sequence
$\{f_{n}\}$ converges uniformly  to $f$ on
$\overline{q^{\mathbb{Z}}}\times S$, then $f \in AP(\overline{q^{\mathbb{Z}}}\times D,\mathbb{X})$.
\end{theorem}
\begin{proof}
For any $\varepsilon>0$, there exists sufficiently large $n_0$ such
that for all $(t,x)\in\overline{q^{\mathbb{Z}}}\times S$,
\begin{equation*}
|f(t,x)-f_{n_0}(t,x)|<\frac{\varepsilon}{3}.
\end{equation*} Take
$q^\tau\in E\big\{f_{n_0},\DF{\varepsilon}{3},S\big\}$, then for all
$(t,x)\in\overline{q^{\mathbb{Z}}}\times S$, we have
{\setlength\arraycolsep{2pt}\begin{eqnarray*}
|f(tq^\tau,x)-f(t,x)|&\leq&|f(tq^\tau,x)-f_{n_0}(tq^\tau,x)|+|f_{n_0}(tq^\tau,x)-f_{n_0}(t,x)|\\
&&+|f_{n_0}(t,x)-f(t,x)|<\varepsilon,
\end{eqnarray*}}
that is, $q^\tau\in E(f,\varepsilon,S)$. Therefore,
 $f \in AP(\overline{q^{\mathbb{Z}}}\times D,\mathbb{X})$. This completes the
proof.
\end{proof}

\begin{theorem}\label{thm316}
Let $f \in AP(\overline{q^{\mathbb{Z}}}\times D,\mathbb{X})$ and
$
F(t,x):=\int_{0}^{t}f(s,x)d_q s,
$
then $F \in AP(\overline{q^{\mathbb{Z}}}\times D,\mathbb{X})$  if
and only if $F$ is bounded on $\overline{q^{\mathbb{Z}}}\times S$.
\end{theorem}
\begin{proof}
If $F \in AP(\overline{q^{\mathbb{Z}}}\times D,\mathbb{X})$, then $F$ is bounded on $\overline{q^{\mathbb{Z}}}\times S$.

If $F$ is bounded, without loss of generality, we can assume that
$F$ is a real-valued function. Denote
$$G:=\sup\limits_{(t,x)\in\overline{q^{\mathbb{Z}}}\times S}F(t,x)>g:=\inf_{(t,x)\in\overline{q^{\mathbb{Z}}}\times S}F(t,x),$$
then for any $\varepsilon>0$, there exist $t_1=q^{n_1}$ and $t_2=q^{n_2}$ such that
\begin{equation*}
F(t_1,x)<g+\frac{\varepsilon}{6},\qquad
F(t_2,x)>G-\frac{\varepsilon}{6},\quad\forall x\in S.
\end{equation*}
Let $\ell=\ell(\varepsilon_1,S)$ be an inclusion length of
$E(f,\varepsilon_1,S)$, where
$\varepsilon_1=\frac{\varepsilon}{6d},d=|t_1-t_2|$. For any
$\alpha\in\overline{q^{\mathbb{Z}}}$, take $\tau\in E(f,\varepsilon_1,S)\cap
[\alpha t_1^{-1}, \alpha t_1^{-1}q^\ell]$. Let $s_i=t_iq^\tau
(i=1,2)$ and $L=\ell+|n_1-n_2|$, then $s_1,s_2\in
[\alpha,\alpha q^L]\cap\overline{q^{\mathbb{Z}}}$. For all $x\in S$, we have
{\setlength\arraycolsep{2pt}\begin{eqnarray*}
F(s_2,x)-F(s_1,x)&=& F(t_2,x)-F(t_1,x)-\int_{t_1}^{t_2}f(t,x)d_q
t+\int_{t_1 q^\tau}^{t_2 q^\tau}f(t,x)d_q t\\
&=& F(t_2,x)-F(t_1,x)+\int_{t_1}^{t_2}[f(tq^\tau,x)-f(t,x)]d_q t\\
&>&G-g-\frac{\varepsilon}{3}-\varepsilon_1d=G-g-\frac{\varepsilon}{2},
\end{eqnarray*}}
that is
\begin{equation*}
(F(s_1,x)-g)+(G-F(s_2,x))<\frac{\varepsilon}{2}.
\end{equation*}
Since
\begin{equation*}
F(s_1,x)-g\geq 0 \quad \mathrm{and} \quad G-F(s_2,x)\geq 0,
\end{equation*}
we have proved that   in any interval  $[\alpha,\alpha q^L]\cap \overline{q^{\mathbb{Z}}}$, there exist $s_1$ and $s_2$ such that
\begin{equation*}
F(s_1,x)<g+\frac{\varepsilon}{2} \quad \mathrm{and} \quad
F(s_2,x)>G-\frac{\varepsilon}{2}.
\end{equation*}

Now, let $\varepsilon_2=\DF{\varepsilon}{2L}$,  we will show that if $q^\tau\in E(f,\varepsilon_2,S)$,
then $q^\tau\in E(F,\varepsilon,S)$. In fact, for any
$(t,x)\in\overline{q^{\mathbb{Z}}}\times S$, we can take $s_1,s_2\in
[t,tq^L]\cap\overline{q^{\mathbb{Z}}}$ such that
\begin{equation*}
F(s_1,x)<g+\frac{\varepsilon}{2}\quad \mathrm{and} \quad
F(s_2,x)>G-\frac{\varepsilon}{2},
\end{equation*}
thus, for $q^\tau\in E(f,\varepsilon_2,S)$, we have
{\setlength\arraycolsep{2pt}\begin{eqnarray*}
F(tq^\tau,x)-F(t,x)&=&F(s_1q^\tau,x)-F(s_1,x)+\int_{t}^{s_1}f(t,x)d_q
t-\int_{tq^\tau}^{s_1q^\tau}f(t,x)d_q t\\
&>&g-(g+\frac{\varepsilon}{2})-\int_{t}^{t_1}[f(tq^\tau,x)-f(t,x)]d_q
t\\
&>&-\frac{\varepsilon}{2}-\varepsilon_2L=-\varepsilon
\end{eqnarray*}}
and
{\setlength\arraycolsep{2pt}\begin{eqnarray*}
F(tq^\tau,x)-F(t,x)&=&
F(s_2q^\tau,x)-F(s_2,x)+\int_{t}^{s_2}f(t,x)d_q
t-\int_{tq^\tau}^{s_2q^\tau}f(t,x)d_q t\\
&<&\frac{\varepsilon}{2}+\varepsilon_2L=\varepsilon.
\end{eqnarray*}}
That is, for $q^\tau\in E(f,\varepsilon_2,S)$,  $q^\tau\in
E(F,\varepsilon,S)$, therefore,  $F \in AP(\overline{q^{\mathbb{Z}}}\times D,\mathbb{X})$. The proof is complete.
\end{proof}

\begin{theorem}\label{thm317}
If $f \in AP(\overline{q^{\mathbb{Z}}}\times D,\mathbb{X})$,
$F(\cdot)$ is uniformly continuous on the value field of $f$,
then $F\circ f \in AP(\overline{q^{\mathbb{Z}}}\times D,\mathbb{X})$.
\end{theorem}
\begin{proof}
Since $F$ is uniformly continuous on the value field of
$f$ and $f\in AP(\overline{q^{\mathbb{Z}}}\times D,\mathbb{X})$, there exists a sequence $\alpha=\{\alpha_n\}\subset \mathbb{Z}$
such that
\begin{equation*}
T_{\alpha}(F\circ
f)=\lim\limits_{n\rightarrow+\infty}F(f(tq^{\alpha_{n}},x))=F(\lim\limits_{n\rightarrow+\infty}f(tq^{\alpha_{n}},x))=F(T_{\alpha}f)
\end{equation*}
holds uniformly on $\overline{q^{\mathbb{Z}}}\times S$. Hence, $F\circ f \in AP(\overline{q^{\mathbb{Z}}}\times D,\mathbb{X})$. The proof is complete.
\end{proof}

\begin{theorem}\label{thm318}
A function $f \in AP(\overline{q^{\mathbb{Z}}}\times D,\mathbb{X})$
if and only if for every pair of sequences
$\alpha^{'},\beta^{'}\subset \mathbb{Z}$, there exist common subsequences
$\alpha\subset\alpha^{'},\beta\subset\beta^{'}$ such that
\begin{equation}\label{e35}
T_{\alpha+\beta}f=T_{\alpha}T_{\beta}f,\quad\forall
(t,x)\in\overline{q^{\mathbb{Z}}}\times S.
\end{equation}
\end{theorem}
\begin{proof}
If $f \in AP(\overline{q^{\mathbb{Z}}}\times D,\mathbb{X})$, then for
every pair of sequences $\alpha^{'},\beta^{'}\subset \mathbb{Z}$, there exists
subsequence $\beta^{(2)}\subset\beta^{'}$ such that
$T_{\beta^{(2)}}f=g\in AP(\overline{q^{\mathbb{Z}}}\times D,\mathbb{X})$
uniformly holds on $\overline{q^{\mathbb{Z}}}\times S$.
One can
take $\alpha^{(2)}\subset\alpha^{'}$ and $\alpha^{(2)} ,\beta^{(2)}$
are the common subsequences of $\alpha^{'},\beta^{'}$, respectively.
Hence, there exists $\alpha^{(3)}\subset\alpha^{(2)}$ such that
$T_{\alpha^{(3)}}g=h \in AP(\overline{q^{\mathbb{Z}}}\times D,\mathbb{X})$
uniformly holds on $\overline{q^{\mathbb{Z}}}\times S$.

Similarly, one can take $\beta^{(3)}\subset\beta^{(2)}$ and
$\beta^{(3)},\alpha^{(3)}$ are the common subsequences of
$\beta^{(2)},\alpha^{(2)}$, respectively. Hence, there exist common
subsequence $\alpha\subset\alpha^{(3)},\beta\subset\beta^{(3)}$ such
that
$T_{\alpha+\beta}f=k \in AP(\overline{q^{\mathbb{Z}}}\times D,\mathbb{X})$
holds uniformly on $\overline{q^{\mathbb{Z}}}\times S$. According to the above, it
is easy to see that
$T_{\beta}f=g$ and $T_{\alpha}g=h$
 hold uniformly on $\overline{q^{\mathbb{Z}}}\times S$.

Hence, for any $\varepsilon>0$, if $n$ is sufficiently large,
then for any $(t,x)\in\overline{q^{\mathbb{Z}}}\times S$, we have
\begin{equation*}
|f(tq^{\alpha_n+\beta_n},x)-k(t,x)|<\frac{\varepsilon}{3},
\,\,
|g(t,x)-f(tq^{\beta_n},x)|<\frac{\varepsilon}{3},
\,\,
|h(t,x)-g(tq^{\alpha_n},x)|<\frac{\varepsilon}{3}.
\end{equation*}
Therefore
{\setlength\arraycolsep{2pt}\begin{eqnarray*}
|h(t,x)-k(t,x)|&\leq&|h(t,x)-g(tq^{\alpha_n},x)|+|g(tq^{\alpha_n},x)-f(tq^{\alpha_n+\beta_n},x)|\\
&&+|f(tq^{\alpha_n+\beta_n},x)-k(t,x)|<\varepsilon
\end{eqnarray*}}
holds for any $(t,x)\in\overline{q^{\mathbb{Z}}}\times S$.  Since $\varepsilon>0$
is arbitrary, we have $h(t,x)\equiv k(t,x)$, that is,
$T_{\alpha+\beta}f=T_\alpha T_\beta f$ holds uniformly on
$\overline{q^{\mathbb{Z}}}\times S$.

On the other hand, if \eqref{e35} holds, then for any sequence
$\gamma^{'}\subset \mathbb{Z}$, there exists subsequence
$\gamma\subset\gamma^{'}$ such that $T_{\gamma}f$ exists
uniformly on $\overline{q^{\mathbb{Z}}}\times S.$

Now, we will prove that $f\in AP(\overline{q^{\mathbb{Z}}}\times D,\mathbb{X})$.
If this is not true, that is, $T_{\gamma}f(t,x)$ does not converge
uniformly on $\overline{q^{\mathbb{Z}}}\times S$, then there exist
$\varepsilon_0>0$,  subsequences
$\alpha^{'}\subset\gamma,\beta^{'}\subset\gamma$ and $s^{'}$.
$\alpha^{'}=\{\alpha_{n}^{'}\},\beta^{'}=\{\beta_{n}^{'}\},s=\{s_{n}^{'}\}$
such that
\begin{equation}\label{e36}
|f(q^{s_{n}^{'}+\alpha_{n}^{'}},x)-f(q^{s_{n}^{'}+\beta_{n}^{'}},x)|\geq\varepsilon_0>0.
\end{equation}
From \eqref{e35} it follows that there exist common subsequences
$\alpha^{(2)}\subset\alpha^{'}, s^{(2)}\subset s^{'}$ such that
\begin{equation}\label{e37}
T_{s^{(2)}+\alpha^{(2)}}f(t,x)=T_{s^{(2)}}T_{\alpha^{(2)}}f(t,x), \,\, \forall (t,x)\in\overline{q^{\mathbb{Z}}}\times S.
\end{equation}
Taking $\beta^{(2)}\subset\beta^{'}$ and
$\beta^{(2)},\alpha^{(2)},s^{(2)}$ are common subsequences of
$\beta^{'},\alpha^{'},s^{'}$ respectively, such that
\begin{equation}\label{e38}
T_{s+\beta}f(t,x)=T_sT_\beta f(t,x), \,\, \forall (t,x)\in\overline{q^{\mathbb{Z}}}\times S.
\end{equation}
Similarly, we can take $\alpha\subset\alpha^{(2)}$ satisfying
$\alpha,\beta,s$ are common subsequences of
$\alpha^{(2)},\beta^{(2)},s^{(2)}$ respectively, by \eqref{e37}, we have
\begin{equation}\label{e39}
T_{s+\alpha}f(t,x)=T_sT_\alpha f(t,x), \,\, \forall (t,x)\in\overline{q^{\mathbb{Z}}}\times S.
\end{equation}
Since $T_\alpha f(t,x)=T_\beta f(t,x)=T_\gamma f(t,x)$, by
\eqref{e38} and \eqref{e39}, for all $(t,x)\in\overline{q^{\mathbb{Z}}}\times S$,
$
T_{s+\beta}f(t,x)=T_{s+\alpha}f(t,x),
$
that is, for all $(t,x)\in\overline{q^{\mathbb{Z}}}\times S$,
$
\lim\limits_{n\rightarrow+\infty}f(tq^{s_n+\beta_n},x)=\lim\limits_{n\rightarrow+\infty}f(tq^{s_n+\alpha_n},x).
$
Taking $t=0$, this contradicts \eqref{e36}. Therefore,
$f \in AP(\overline{q^{\mathbb{Z}}}\times D,\mathbb{X})$. The proof
is complete.
\end{proof}

\section{Almost periodic dynamic equations on the quantum time scale}
\setcounter{equation}{0}\indent

Consider the following nonlinear quantum dynamic equation
\begin{equation}\label{e41}
D_qx =f(t,x),
\end{equation}
where $f\in C({\overline{q^{\mathbb{Z}}}}\times\mathbb{X}^{n},\mathbb{X}^{n})$.

Now, we establish a existence theorem for almost
periodic solution of \eqref{e41} by using Theorem \ref{thm38}.

\begin{theorem}\label{as}
If $\phi$ is an asymptotically almost periodic solution of \eqref{e41}
on $q^{\mathbb{Z}^+}$, then the almost periodic part of $\phi$ is   an almost periodic solution of \eqref{e41}.
\end{theorem}
\begin{proof}
Let $\phi=p+q$,
where $p$ is the almost periodic part and $q(t)\rightarrow 0$ as
$t\rightarrow\infty$.  By Theorem \ref{thm38}, for any sequence $\alpha'=\{\alpha'_n\}\subset \mathbb{Z}$, there exists a subsequence $\alpha\subset
\alpha'$ such that $T_{\alpha'}f=g$ holds uniformly on
$\overline{q^\mathbb{Z}}\times S$ and $T_\alpha p=\psi$  holds uniformly on $\overline{q^\mathbb{Z}}$.
Therefore, $T_\alpha\phi=T_\alpha p=\psi$ is an almost periodic
solution of the corresponding equation in the hull
\[D_qx(t)=g(t,x)\]
on $\mathbb{T}$.
 Now that $T_{-\alpha}g=f$ holds uniformly on $\overline{q^\mathbb{Z}}\times S$ and
$T_{-\alpha}\psi=p$ holds uniformly on $\overline{q^\mathbb{Z}}$, hence $p$ is an almost
periodic solution of \eqref{e41}. The proof is complete.
\end{proof}

\begin{theorem}
If for each $g\in H(f)$, the equation $D_qx (t)=g(t,x(t))$ in the hull has
a solution in $S$, where $S\subset \mathbb{E}^n$ is a compact set, then these solutions are almost periodic.
In particular, system \eqref{e41} has an almost periodic solution in $S$.
\end{theorem}
\begin{proof}
Let $\phi$ be a solution of $D_qx (t)=g(t,x(t))$ in
$S$ with $g\in H(f)$. For any given sequences $\alpha'\subset\mathbb{Z}$ and $\beta'\subset\mathbb{Z}$, we can extract
common subsequences  $\alpha \subset \alpha'$ and $\beta\subset \beta'$ such that $T_{\alpha+\beta}f(t,x)=T_\alpha T_\beta f(t,x)$ holds uniformly on $\overline{q^\mathbb{Z}}\times S$, and $T_{\alpha+\beta}\varphi$ and $T_\alpha T_\beta \varphi$ exist  on any compact subset of $\overline{q^\mathbb{Z}}$.
Therefore, both $T_{\alpha+\beta}\varphi(t)$ and $T_\alpha T_\beta \varphi(t)$ are solutions of the same equation in the hull
\[
D_qx(t)=T_{\alpha+\beta}f(t,x)
\]
in $S$. Hence, $T_{\alpha+\beta}\varphi=T_\alpha T_\beta \varphi$, so, $\varphi$ is an almost periodic solution.
The proof is complete.
\end{proof}

For system \eqref{e41}, we
consider its product system:
\begin{equation*}\label{ass}
x^\Delta=f(t,x),\,\,\,\,\,y^\Delta=f(t,y).
\end{equation*}
\begin{theorem}
Suppose that there exists a Lyapunov function $V(t,x,y)\in
C(q^{\mathbb{Z}^+}\times D\times D, \mathbb{R})$ satisfying the following
conditions:
\begin{itemize}
  \item [$(i)$] $a(|x-y|)\leq V(t,x,y)\leq b(|x-y|)$, where $a,b\in\mathcal
  {K}$ with $\mathcal{K}=\{a\in C(\mathbb{R}^+, \mathbb{R}^+):a(0)=0$ and $a$ is
  increasing \};
  \item [$(ii)$] $|V(t,x_1,y_1)-v(t,x_2,y_2)|\leq
  L[|x_1-x_2|+|y_1-y_2]$, where $L>0$ is a constant;
  \item [$(iii)$] $D^+V^\Delta_{(4.3)}(t,x,y)\leq -cV(t,x,y)$
  where $-c\in\mathcal{R}^+$ and $c>0$.
\end{itemize}

Moreover, if there exists a solution $x(t)$ of (4.1) such that
$x(t)\in S$ for $t\in q^{\mathbb{Z}^+}$, where $S\subset D$ is a compact set.
Then there exists a unique uniformly asymptotically stable almost
periodic solution $p(t)$ of (4.1) in $S$. Furthermore, if $f(t,x)$
is periodic with period $\omega$ in $t$, then $p(t)$ is a periodic
solution of (4.1) with period $\omega$.
\end{theorem}
\begin{proof}
Let $\alpha=\{\alpha_n\}\subset\mathbb{Z}$ be a sequence such that
$\a_n\rightarrow\infty$ as $n\rightarrow\infty$ and $T_\alpha
f=f$ uniformly on $\overline{q^\mathbb{Z}}\times S$. Assume that
$\phi(t)\subset S$ is a solution of \eqref{e41} for $t\in q^{\mathbb{Z}^+}$. Then
$\phi(tq^{\a_n})$ is a solution of the equation
\[D_qx=f(tq^{\a_n},x),\]
which is also in $S$. For a given $\e>0$, choose an integer
$k_0(\e)$ so large that if $m\geq k>k_0(\e)$, we have
\begin{equation}\label{as1}
b(2B)e_{-c}(\a_k,0)<a(\e)/2
\end{equation}
and
\begin{equation}\label{as2}
|f(tq^{\a_k},x)-f(tq^{\a_m},x)|<\frac{ca(\e)}{2L},
\end{equation}
where $B$ is a positive constant such that $S\subset\{x:|x|\leq
B\}$. Then it follows from $(ii)$ and $(iii)$ that
\begin{eqnarray*}
D^+V^\Delta (t,\phi(t), \phi(tq^{-\alpha_m-\alpha_k}))&\leq&
-cV(t,\phi(t),\phi(tq^{-\alpha_m-\alpha_k}))\\
&&+L|f(tq^{\alpha_m-\alpha_k},\phi(tq^{\alpha_m-\alpha_k)})-f(t,\phi(tq^{\alpha_m-\alpha_k})))|.
\end{eqnarray*}
In view of \eqref{as2}, we have
\begin{equation}\label{as3}
D^+V^\Delta (t,\phi(t), \phi(tq^{-\alpha_m-\alpha_k}))\leq
-cV(t,\phi(t),\phi(tq^{-\alpha_m-\alpha_k}))+\frac{ca(\e)}{2},
\end{equation}
 If $m\geq
k\geq k_0(\e)$, by Lemma \ref{lmas}, condition (i) and \eqref{as1}, \eqref{as3} implies
that
\begin{eqnarray*}
V (tq^{\alpha_k},\phi(tq^{\alpha_k}), \phi(tq^{-\alpha_m}))&\leq&
e_{-c}(tq^{\alpha_m},0)V(0,\phi(0),\phi(q^{\alpha_m-\alpha_k}))\\
&&+\frac{a(\e)}{2}(1-e_{-c}(tq^{\alpha_k},0))\\
&\leq&
e_{-c}(tq^{\a_m},0)V(0,\phi(0),\phi(q^{\alpha_m-\alpha_k}))+\frac{a(\e)}{2}\\
&<&a(\e)\,\,,\,\,\,\,{\rm for}\,\,t\in q^{\mathbb{Z}^+}.
\end{eqnarray*}
Therefore, by condition $(i)$, for $m\geq k\geq k_0$ and $ t\in q^{\mathbb{Z}^+}$, we have
\[
|\phi(tq^{\alpha_k})-\phi(tq^{\alpha_m})|<\varepsilon,
\]
which implies that $\phi$ is asymptotically almost periodic. Hence, it follows from Theorem \ref{as}  that the almost periodic part $p(t)\subset S$ of $\phi$
is an almost periodic solution of system \eqref{e41}.

By using the Lyapunov function $V(t,x,y)$,  it is easy to show that $p(t)$ is uniformly
asymptotically stable and every solution remaining in $D$ approaches
$p(t)$ as $t\rightarrow\infty$, which also implies the uniqueness of
$p(t)$.

In particular, if $f(t,x)$ is periodic in $t$ with period $\omega$, then
$p(tq^\omega)$ is also a solution of \eqref{e41} which remains in $S$. By the
uniqueness of the solutions, we know that $p(tq^\omega)=p(t)$. This
completes the proof.
\end{proof}

\section{An equivalent definition of almost periodic functions on the quantum time scale}
\setcounter{equation}{0}

\indent

In this section, we will give an equivalent definition of almost periodic functions on the quantum time scale $\overline{q^\mathbb{Z}}$.
To this end,
we introduce a notation $-\infty_q$ and  stipulate $q^{-\infty_q}=0, t\pm(-\infty_q)=t$ and $t>-\infty_q$ for any $t\in \mathbb{Z}$.
Let $f\in C(\overline{q^\mathbb{Z}},\mathbb{X})$, we define a function $\tilde{f}:  \mathbb{Z}\cup \{-\infty_q\}\rightarrow \mathbb{X}$ by
\begin{eqnarray}\label{g1}
\tilde{f}(t)=\left\{\begin{array}{lll}
f(q^t),& t\in \mathbb{Z},\\
f(0),& t=-\infty_q,\end{array}\right.
\end{eqnarray}
that is,
\begin{eqnarray*}
f(t)=\left\{\begin{array}{lll}
\tilde{f}(\log_q t),& t\in q^\mathbb{Z},\\
\lim\limits_{t\rightarrow 0^+}f(t),& t=0.\end{array}\right.
\end{eqnarray*}
 Since $f(t)$ is right continuous at $t=0$, it is clear that the above definition is well defined.

Moreover, for $f\in C(\overline{q^\mathbb{Z}}\times \mathbb{X},\mathbb{X})$, we define a function $\tilde{f}:  \mathbb{Z}\cup \{-\infty_q\}\times \mathbb{X}\rightarrow \mathbb{X}$ by
\begin{eqnarray}\label{g2}
\tilde{f}(t,x)=\left\{\begin{array}{lll}
f(q^t,x),& (t,x)\in \mathbb{Z}\times \mathbb{X},\\
f(0,x),& t=-\infty_q, x\in \mathbb{X},\end{array}\right.
\end{eqnarray}
that is,
\begin{eqnarray*}
f(t,x)=\left\{\begin{array}{lll}
\tilde{f}(\log_q t,x),& (t,x)\in q^\mathbb{Z}\times \mathbb{X},\\
\lim\limits_{t\rightarrow 0}f(t,x),& t=0, x\in \mathbb{X}.\end{array}\right.
\end{eqnarray*}
Since $f(t,x)$ is continuous at $(0,x)$, it is clear that the above definition is well defined.

\begin{definition}\label{feq2}
A  function $f\in C( (\mathbb{Z}\cup\{-\infty_q\})\times \mathbb{X}, \mathbb{X})$ is
called almost periodic in $t\in \mathbb{Z}\cup\{-\infty_q\}$ uniformly in $x\in \mathbb{X}$ if for every $\varepsilon>0$ and each compact subset $S$ of $\mathbb{X}$
there exists an  integer $l(\varepsilon)>0$
such that each interval of
 length $l(\varepsilon,S)$ contains an integer $\tau$ such that
\begin{equation}\label{eq2b}
|f(t+\tau,x)-f(t,x)|<\varepsilon, \,\, \forall (t,x) \in  (\mathbb{Z}\cup\{-\infty_q\})\times S.
\end{equation}
This $\tau$ is called the $\varepsilon$-translation number or $\varepsilon$-almost period.
\end{definition}

\begin{definition}\label{feq1}
A function $ f\in C(\mathbb{Z}\cup\{-\infty_q\},\mathbb{X})$ is called almost periodic
if for every  $\varepsilon>0$,
 there exists an  integer $l(\varepsilon)>0$
such that each interval of
 length $l(\varepsilon)$ contains an integer $\tau$ such that
\begin{equation}\label{eq1a}
|f(t+\tau)-f(t)|<\varepsilon, \,\, \forall t \in  \mathbb{Z}\cup\{-\infty_q\}.
\end{equation}
This $\tau$ is called the $\varepsilon$-translation number or $\varepsilon$-almost period.
\end{definition}

\begin{remark}\label{re31a}
According to the operation rule about $-\infty_q$, $\mathbb{Z}\cup\{-\infty_q\}$ is an almost periodic time scale under the Definition \ref{li67} and when $t=-\infty_q$, the inequalities \eqref{eq2b} and \eqref{eq1a} hold naturally. Hence, $\forall t \in  \mathbb{Z}\cup\{-\infty_q\}$ and
$\forall (t,x) \in  (\mathbb{Z}\cup\{-\infty_q\})\times S$ in \eqref{eq2b} and \eqref{eq1a} can be replaced by $\forall t \in  \mathbb{Z}$ and
$\forall (t,x) \in  \mathbb{Z}\times S$, respectively.
\end{remark}

\begin{remark}\label{re31}
Obviously, the almost periodic functions defined by Definitions \ref{feq2} and \ref{feq1} (which are defined on $\mathbb{Z}\cup\{-\infty_q\}\times \mathbb{X}$) or $\mathbb{Z}\cup\{-\infty_q\}$ have the same properties as  the ordinary almost periodic functions  defined on $\mathbb{Z}\times \mathbb{X}$ or $\mathbb{Z}$, respectively.
\end{remark}

\begin{definition}\label{bc2}
A  function   $f\in C(\overline{q^\mathbb{Z}}\times \mathbb{X},\mathbb{X})$ is
called almost   periodic in  $t \in \overline{q^\mathbb{Z}}$ for each $x \in \mathbb{X}$ if and only if  the function $\tilde{f}(t,x)$ defined by \eqref{g2} is almost   periodic in  $t \in \overline{q^\mathbb{Z}}$ for each $x \in \mathbb{X}$.
\end{definition}

\begin{definition}\label{bc1}
A  function   $f\in C(\overline{q^\mathbb{Z}},\mathbb{X})$ is
called almost   periodic if and only if  the function $\tilde{f}(t)$ defined by \eqref{g1} is almost periodic.
\end{definition}

Obviously, Definitions \ref{bc2} and \ref{bc1} are equivalent to Definitions \ref{def34} and \ref{def34y}, respectively.
Moreover, by Remark \ref{re31}, all of the properties of almost periodic functions on the quantum time scale can be directly obtained from the corresponding properties of the ordinary almost periodic functions defined on $\mathbb{Z}$ or $\mathbb{Z}\times \mathbb{X}$.

\section{An application}
\setcounter{equation}{0}\indent

Based on the equivalent definition of almost periodic functions on the quantum time scale, Remark \ref{re31a}, the transformations \eqref{g1} and \eqref{g2},
we can transform the almost periodic problem of dynamic equations on the quantum time scale into the almost periodic problem of dynamic equations on almost periodic time scales.
For example, consider the dynamic equation on the  quantum time scale:
\begin{equation}\label{qqn1}
D_qx  (t)=f(t,x(t),x(tq^{-\sigma(t)}))),\,\, t\in \overline{q^\mathbb{Z}},
\end{equation}
where $\sigma: \mathbb{Z} \rightarrow \mathbb{Z}$ is a scalar delay function,  $f\in C(\overline{q^\mathbb{T}}\times \mathbb{E}^{2n},\mathbb{E}^n)$.
 Under the transformation \eqref{g2}, equation \eqref{qqn1} is transformed to
\begin{equation}\label{qqn2}
\Delta \tilde{x}  (n)=F(n,\tilde{x}(n),\tilde{x}(n-\tilde{\sigma}(n))),\,\,n\in \mathbb{Z}\cup\{-\infty_q\},
\end{equation}
and vice visa, where $F(n)=(q-1)q^n\tilde{f}(n,\tilde{x}(n),\tilde{x}(n-\tilde{\sigma}(n)))).$

Clearly, if $x(t)$ is a solution of \eqref{qqn1} if and only if  $\tilde{x}(n)$ is a solution of \eqref{qqn2}.
Therefore, applying the theory of difference equations or the theory of almost periodic dynamic equations on time scales (\cite{bing5}), we can first discuss the  almost periodic problem of \eqref{qqn2}, and then by  transformation \eqref{g2} we  can obtain the corresponding result about \eqref{qqn1}. Besides, since when $n=\{-\infty\}_q$, both sides of \eqref{qqn2} are equal to $0$, so, if $\phi(n)$ is a solution of the following equation:
\begin{equation*}\label{qqn3}
\Delta \tilde{x}  (n)=F(n,\tilde{x}(n),\tilde{x}(n-\tilde{\sigma}(n))),\,\,n\in \mathbb{Z},
\end{equation*}
then
\begin{equation*}
\varphi(n)=\left\{\begin{array}{lll}\phi(n),&n\in \mathbb{Z},\\
\phi(0),&n=\{-\infty\}_q\end{array}\right.
\end{equation*}is a solution of \eqref{qqn2}.

As an application, we consider the
following  high-order Hopfield neural network on  the  quantum time scale $\overline{q^\mathbb{T}}$:
\begin{eqnarray}\label{qqe11}
 D_q x_i(t)&=&-c_i(t)x_i(t)+
\sum^m_{j=1}a_{ij}(t)f_j(x_j(tq^{-\gamma_{ij}(t)}))\nonumber\\
&&+\sum^m_{j=1}\sum_{l=1}^{m}b_{ijl}(t)
g_j(x_j(tq^{-\omega_{ijl}(t)}))g_l(x_l(tq^{-v_{ijl}(t)}))+I_i(t),\,\,t\in \overline{q^\mathbb{T}},
\end{eqnarray}
where $i=1,2,\ldots,m$, $m$ corresponds to
the number of units in a neural network, $x_i(t)$ corresponds to the
state vector of the $i$th unit at the time $t$, $c_i(t)$ represents
the rate with which the $i$th unit will reset its potential to the
resting state in isolation when disconnected from the network and
external inputs at time $t$, $a_{ij}(t),\alpha_{ij}(t)$ and
$b_{ijl}(t)$ are the first- and second-order connection weights of
the neural network at time $t$, $I_i(t)$ denotes the external inputs at time $t$,
$f_j$ and $g_j$ are the activation functions of signal
transmission, $\gamma_{ij}(t),\omega_{ijl}(t),v_{ijl}(t): \mathbb{Z}\rightarrow \mathbb{Z}^+$
correspond to the transmission delays at time $t$.

 Under the transformation \eqref{g2}, equation \eqref{qqe11} is transformed to
\begin{eqnarray}\label{qqe112}
 \Delta \tilde{x}_i(n)&=&-(q-1)q^n\tilde{c}_i(n)\tilde{x}_i(n)+
(q-1)q^n\sum^m_{j=1}\tilde{a}_{ij}(n)\tilde{f}_j(\tilde{x}_j(n-\tilde{\gamma}_{ij}(n)))\nonumber\\
&&+(q-1)q^n\sum^m_{j=1}\sum_{l=1}^{m}\tilde{b}_{ijl}(n)
\tilde{g}_j(\tilde{x}_j(n-\tilde{\omega}_{ijl}(n)))\tilde{g}_l(\tilde{x}_l(n-\tilde{v}_{ijl}(n)))\nonumber\\
&&+(q-1)q^n\tilde{I}_i(n),\,\,n\in \mathbb{Z}\cup\{-\infty_q\},\,i=1,2,\ldots,m
\end{eqnarray}
and vice visa. If we take $\mathbb{T}=\mathbb{Z}\cup\{-\infty_q\}$, then \eqref{qqe112} can be written as
\begin{eqnarray}\label{qqe1121}
 \tilde{x}^\Delta_i(t)&=&-\hat{c}_i(t)\tilde{x}_i(t)+
\sum^m_{j=1}\hat{a}_{ij}(t)\tilde{f}_j(\tilde{x}_j(t-\tilde{\gamma}_{ij}(t)))\nonumber\\
&&+\sum^m_{j=1}\sum_{l=1}^{m}\hat{b}_{ijl}(t)
\tilde{g}_j(\tilde{x}_j(t-\tilde{\omega}_{ijl}(t)))\tilde{g}_l(\tilde{x}_l(t-\tilde{v}_{ijl}(t)))\nonumber\\
&&+ \hat{I}_i(t),\,\,t\in \mathbb{T},\,i=1,2,\ldots,m,
\end{eqnarray}
where $c_i(t)=(q-1)q^t\tilde{c}_i(t),\hat{a}_{ij}(t)=(q-1)q^t\tilde{a}_{ij}(t),\hat{b}_{ijl}(t)=(q-1)q^t\tilde{b}_{ijl}(t), \hat{I}_i(t)=(q-1)q^t\tilde{I}_i(t),t\in \mathbb{T},\,i,j,l=1,2,\ldots,m.$
From the discussion above, we see that if \eqref{qqe1121} has an almost periodic solution $\tilde{x}=(\tilde{x}_1,\tilde{x}_2,\ldots,\tilde{x}_m)^T$, then \eqref{qqe11} has an almost periodic solution $x=(x_,x_2,\ldots,x_m)^T$.

Now, we shall establish the existence of almost periodic solutions of \eqref{qqe1121}.
First, for convenience, we introduce some notations. We will use
$x=(x_1,x_2,\ldots,x_n)^T\in\mathbb{R}^n$ to denote a column vector,
in which the symbol $T$ denotes the transpose of vector. We let
$|x|$ denote the absolute-value vector given by
$|x|=(|x_1|,|x_2|,\ldots,|x_n|)^T$, and define
$\|x\|=\max\limits_{1\leq i\leq n}|x_i|$.

Let $AP(\mathbb{T})=\{x\in C(\mathbb{T},\mathbb{R}):x(t)$ is a
real valued, almost periodic function on $\mathbb{T}\}$ with the norm
$\|\varphi\|_0=\sup\limits_{t\in\mathbb{T}}\|\varphi(t)\|_0$,
where
$\|\varphi(t)\|_0=\max\limits_{1\leq i\leq
n}|\varphi_i(t)|$
and $\varphi \in AP(\mathbb{T})$, then
$AP(\mathbb{T})$ is an Banach space.

\begin{theorem}\label{th31} Assume that
\begin{itemize}
\item [$(H_1)$] $\tilde{f}_j,\tilde{g}_j\in C(\mathbb{R},\mathbb{R})$, and there
    exist positive constants
    $\epsilon_j$ and  $\varepsilon_j$ such that
    \[
      |\tilde{f}_j(u)-\tilde{f}_j(v)|\leq\epsilon_j|u-v|,\,|\tilde{g}_j(u)-\tilde{g}_j(v)|\leq
      \varepsilon_j|u-v|,
    \]
    where $|u|,|v|\in\mathbb{R}, j=1,2,\ldots,m$;
    \item [$(H_2)$]  there exist constants $N_j>0$ such that
    $|\tilde{g}_j(u)|\leq N_j, u\in\mathbb{R}, j=1,2,\ldots,m$;
\item [$(H_3)$] for $i,j,l=1,2,\ldots,m$, $\hat{c}_i$ is positive regressive and $\inf_{t\in\mathbb{T}}
\hat{c}_i(t)>0, \hat{c}_i,\hat{a}_{ij}, \hat{b}_{ijl},\hat{I}_i,\\
\tilde{\gamma}_{ij},\tilde{\omega}_{ijl},\tilde{v}_{ijl}\in
AP(\mathbb{T})$;
\item [$(H_4)$] there exists a constant $r_0$ such that
\[
\max_{1\leq i\leq
m}\bigg\{\frac{\eta_i}{\hat{c}^-_i}\bigg\}+L\leq r_0,\,\,
\max\limits_{1\leq i\leq
m}\{\bar{\eta}_i\} <\min\limits_{1\leq i\leq m}\{
\hat{c}^-_i\},
\]
 where, for
$i,j,l=1,2,\ldots,n$,
\[
\eta_i=\sum^m_{j=1}\bigg[\hat{a}^+_{ij}\big(|f_j(0)|+\epsilon_{j}r_0\big) \big)
+\big(|g_j(0)|+\varepsilon_{j}r_0\big)\sum^m_{l=1}\hat{b}^+_{ijl}(|g_l(0)|+\varepsilon_{l}r_0\big)\bigg],
\]
\[
\bar{\eta}_i=\sum^m_{j=1}\hat{a}^+_{ij}\epsilon_j+\sum^m_{j=1}\sum^m_{l=1}
\hat{b}^+_{ijl}(N_l\varepsilon_j+N_j\varepsilon_l),\,\,
L=\max_{1\leq i\leq
n}\frac{\overline{I_i}}{\hat{c}^-_i},\,\, \hat{c}^+_i=\sup_{t\in\mathbb{T}}c_i(t),
\]
\[
\hat{c}^-_i=\inf_{t\in\mathbb{T}}
\hat{c}_i(t),\,\, \hat{a}^+_{ij}=\sup_{t\in\mathbb{T}}|\hat{a}_{ij}(t)|,\,\,  \hat{b}^+_{ijl}
=\sup_{t\in\mathbb{T}}|\hat{b}_{ijl}(t)|,\,\, \hat{I}^+_i=\sup_{t\in\mathbb{T}}|\hat{I}_i(t)|.
\]
\end{itemize}
Then system \eqref{qqe1121} has a unique almost periodic solution in the
region
\[
E=\{\varphi\in AP(\mathbb{T}):\|\varphi\|_0\leq r_0\}.
\]
\end{theorem}
\begin{proof}
For any given $\varphi\in AP(\mathbb{T})$, we consider the following
almost dynamic equation
\begin{eqnarray}\label{ze31}
\tilde{x}_i^\Delta(t)&=&-\hat{c}_i(t)\tilde{x}_i(t)+\sum^m_{j=1}\hat{a}_{ij}(t)\tilde{f}_j(\varphi_j(t-\tilde{\gamma}_{ij}(t)))\nonumber\\
&&+\sum^m_{j=1}\sum^m_{l=1}
\hat{b}_{ijl}(t)\tilde{g}_j(\varphi_j(t-\tilde{\omega}_{ijl}(t)))\tilde{g}_l(\varphi_l(t-\tilde{v}_{ijl}(t)))+\hat{I}_i(t),\,\, i=1,2,\ldots,m.\quad
\end{eqnarray}
In view of $(H_3)$,
it follows from Lemma 2.15 in \cite{ddns} that the linear system
\[
\tilde{x}_i^{\Delta}(t)=-\hat{c}_i(t)\tilde{x}_i(t),\,\, i=1,2,\ldots,m
\]
admits an exponential dichotomy on $\mathbb{T}$. Thus, by Theorem 4.19 in \cite{bing5},
we know that system \eqref{ze31} has exactly one almost periodic solution
\begin{eqnarray*}
(\tilde{x}_{\varphi})_i(t)&=&\int^t_{-\infty}e_{-\hat{c}_i}(t,\sigma(s))\bigg(\sum^m_{j=1}\bigg[\hat{a}_{ij}(s)
\tilde{f}_j(\varphi_j(s-\tilde{\gamma}_{ij}(s)))\\
&&+\sum^m_{j=1}\sum^m_{l=1}\hat{b}_{ijl}(s)\tilde{g}_j(\varphi_j(s-\tilde{\omega}_{ijl}(s)))\tilde{g}_l(\varphi_l(s-\tilde{v}_{ijl}(s)))+\hat{I}_i(s)
\bigg]\Delta s,\,\,i=1,2,\ldots,m.
\end{eqnarray*}
First, we define a nonlinear operator on $AP(\mathbb{T})$ by
\[
\Phi(\varphi)(t)=\tilde{x}_{\varphi}(t),\,\,\,\forall \varphi\in AP(\mathbb{T}).
\]
Next, we check that $\Phi(E)\subset E$. For any given $\varphi\in
E$, it suffices to prove that $\|\Phi(\varphi)\|_0\leq
r_0$. By conditions $(H_1)$-$(H_4)$, we have
\begin{eqnarray*}
\|\Phi(\varphi)\|_0&=&\sup_{t\in\mathbb{T}}\max_{1\leq i\leq
m}\bigg\{\bigg|\int^t_{-\infty}e_{-\hat{c}_i}
(t,\sigma(s))\bigg(\sum^m_{j=1}\hat{a}_{ij}(s)\tilde{f}_j(\varphi_j(s-\tilde{\gamma}_{ij}(s)))\nonumber\\
&&+\sum^m_{j=1}\sum^m_{l=1}\hat{b}_{ijl}(s)\tilde{g}_j(\varphi_j(s-\tilde{\omega}_{ijl}(s)))\tilde{g}_l(\varphi_l(s-\tilde{v}_{ijl}(s)))
+\hat{I}_i(s)\bigg)\Delta s\bigg|\bigg\}\nonumber\\
&\leq&\sup_{t\in\mathbb{T}}\max_{1\leq i\leq
m}\bigg\{\bigg|\int^t_{-\infty}e_{-\hat{c}^-_i}(t,\sigma(s))\bigg(
\sum^m_{j=1}\hat{a}^+_{ij}\tilde{f}_j(\varphi_j(s-\tilde{\gamma}_{ij}(s)))\nonumber\\
&&+\sum^m_{j=1}\sum^m_{l=1}\hat{b}^+_{ijl}\tilde{g}_j(\varphi_j(s-\tilde{\omega}_{ijl}(s)))\tilde{g}_l(\varphi_l(s-\tilde{v}_{ijl}(s)))\bigg)\Delta
s\bigg|\bigg\}+\max_{1\leq i\leq n}\frac{\hat{I}^+_i}{\hat{c}^-_i}\nonumber\\
&\leq&\sup_{t\in\mathbb{T}}\max_{1\leq i\leq
m}\bigg\{\bigg|\int^t_{-\infty}e_{-\hat{c}^-_i}(t,\sigma(s))\bigg(\sum^m_{j=1}
\hat{a}^+_{ij}\big(|\tilde{f}_j(0)|+\epsilon_j|\varphi_j(s-\tilde{\gamma}_{ij}(s))|\big)\nonumber\\
&&+\sum^m_{j=1}\sum^m_{l=1}\hat{b}^+_{ijl}\big(|\tilde{g}_j(0)|+\varepsilon_j|\varphi_j(s-\tilde{\omega}_{ijl}(s))|\big)
\big(|\tilde{g}_l(0)|+\varepsilon_l|\varphi_l(s-\tilde{v}_{ijl}(s))|\big)\bigg)\Delta
s\bigg|\bigg\}+L\nonumber\\
&\leq&\sup_{t\in\mathbb{T}}\max_{1\leq i\leq
m}\bigg\{\bigg|\int^t_{-\infty}e_{-\hat{c}^-_i}(t,\sigma(s))\bigg(\sum^m_{j=1}
\hat{a}^+_{ij}\big(|\tilde{f}_j(0)|+\epsilon_{j}r_0\big)\nonumber\\
&&+\sum^m_{j=1}\sum^m_{l=1}\hat{b}^+_{ijl}
\big(|\tilde{g}_j(0)|+\varepsilon_{j}r_0\big)
\big(|\tilde{g}_l(0)+\varepsilon_{l}r_0\big)\bigg)\Delta s\bigg|\bigg\}+L\nonumber\\
&\leq&\max_{1\leq i\leq
m}\bigg\{\frac{\eta_i}{\hat{c}^-_i}\bigg\}+L\leq r_0,
\end{eqnarray*}
which implies that $\Phi(E)\subset E$.
For any $\varphi,\psi\in E$, according to $(H_1)$ and
$(H_4)$, we have
\begin{eqnarray*}
&&\|\Phi(\varphi)-\Phi(\psi)\|_0\nonumber\\
&=&\sup_{t\in\mathbb{T}}\max_{1\leq i\leq
m}\bigg\{\bigg|\int^t_{-\infty}e_{-\hat{c}_i}(t,\sigma(s))\bigg(\sum^m_{j=1}
\hat{a}_{ij}(s)\big[\tilde{f}_j(\varphi_j(s-\tilde{\gamma}_{ij}(s)))-\tilde{f}_j(\psi_j(s-\tilde{\gamma}_{ij}(s)))\big]\nonumber\\
&&+\sum^m_{j=1}\sum^m_{l=1}\hat{b}_{ijl}(s)\tilde{g}_j(\varphi_j(s-\tilde{\omega}_{ijl}(s)))
\tilde{g}_l(\varphi_l(s-\tilde{v}_{ijl}(s)))\nonumber\\
&&-\sum^m_{j=1}b_{ijl}(s)\tilde{g}_j(\psi_j(s-\tilde{\omega}_{ijl}(s)))
\tilde{g}_l(\psi_l(s-\tilde{v}_{ijl}(s)))\nonumber\\
&\leq&\sup_{t\in\mathbb{T}}\max_{1\leq i\leq
m}\bigg\{\int^t_{-\infty}e_{-\hat{c}^-_i}(t,\sigma(s))\bigg(
\sum^m_{j=1}\hat{a}^+_{ij}\epsilon_j|\varphi_j(s-\tilde{\gamma}_{ij}(s))-\psi_j(s-\tilde{\gamma}_{ij}(s))|\nonumber\\
&&+\sum^m_{j=1}\sum^m_{l=1}\hat{b}^+_{ijl}\varepsilon_j|\varphi_j(s-\tilde{\omega}_{ijl}(s))-
\psi_j(s-\tilde{\omega}_{ijl}(s))||\tilde{g}_l(\varphi_l(s-\tilde{v}_{ijl}(s)))|\nonumber\\
&&+\sum^m_{j=1}\sum^m_{l=1}\hat{b}^+_{ijl}\varepsilon_l|\varphi_l(s-\tilde{v}_{ijl}(s))
-\psi_l(s-\tilde{v}_{ijl}(s))||\tilde{g}_j(\psi_j(s-\tilde{\omega}_{ijl}(s)))|\bigg)\Delta s\bigg\}\nonumber\\
&\leq&\sup_{t\in\mathbb{T}}\max_{1\leq i\leq
m}\bigg\{\int^t_{-\infty}e_{-\hat{c}^-_i}(t,\sigma(s))\bigg(\sum^m_{j=1}
\hat{a}^+_{ij}\epsilon_j
+\sum^m_{j=1}\sum^m_{l=1}\hat{b}^+_{ijl}(N_j\varepsilon_{l}+N_l\varepsilon_j)\bigg)\|\varphi-\psi\|_{0}\Delta
s\bigg\}\nonumber\\
&\leq&\frac{\max\limits_{1\leq i\leq
m}\{\bar{\eta}_i\}}{\min\limits_{1\leq i\leq m}\{
\hat{c}^-_i\}}\|\varphi-\psi\|_{0},
\end{eqnarray*}
which
implies that $\Phi$ is a contraction
mapping. Hence, $\Phi$ has a unique fixed point in $E$, that is
\eqref{qqe1121} has a unique almost periodic solution in the region
$E$.
This completes the proof.
\end{proof}

\section{Conclusion}
\indent

In this paper, we proposed two types of concepts of almost periodic functions on the quantum time scale and investigated some their basic properties. Moreover, we established a result about the existence and uniqueness of almost
periodic solutions of  dynamic equations on the quantum time scale by Lyapunov method and  gave an equivalent concept of  almost periodic functions on the quantum time scale. In addition, as an application, we proposed a class of high-order Hopfield neural network on  the  quantum time scale and established the existence of almost periodic solutions of the neural networks.
Based on the concepts introduced and the results obtained in this paper, we can further study almost periodic problems for dynamic equations on the quantum time scale. For example, we can study the almost periodicity of  population dynamical models on the quantum time scale. Furthermore,
by using the idea of transformations of \eqref{g1} and \eqref{g2}, we can give a new concept of  almost periodic functions on  time scales as follows:
\begin{definition}\label{final}
Let $\mathbb{T}$ be a time scale, $\tilde{\mathbb{T}}$ be an almost periodic time scale defined by Definition \ref{li67} and there exist a one to one mapping $\chi: \mathbb{T}\rightarrow \mathbb{R}$ such that $\chi(\mathbb{T})=\tilde{\mathbb{T}}$. A function $f\in
C(\mathbb{T}\times D, E^n)$ is called an almost periodic function in
$t\in\mathbb{T}$ uniformly in $x\in D$ if the
$\varepsilon$-translation set of $f$
$$E\{\varepsilon,f,S\}=\{\tau\in\Pi:|f(\chi^{-1}(t+\tau),x)-f(\chi^{-1}(t),x)|<\varepsilon,\,\,
\forall\,\, (t,x)\in  \tilde{\mathbb{T}}\times S\}$$ is  relatively
dense  for every $\varepsilon>0$ and   for each
compact subset $S$ of $D$; that is, if for every $\varepsilon>0$
and for each compact subset $S$ of $D$, there exists
a constant $\ell(\varepsilon,S)>0$ such that each interval of length
$\ell(\varepsilon,S)$ contains a $\tau(\varepsilon,S)\in
E\{\varepsilon,f,S\}$ such that
\[
|f(\chi^{-1}(t+\tau),x)-f(\chi^{-1}(t),x)|<\varepsilon,\,\,\, \forall (t,x)\in
\tilde{\mathbb{T}}\times S,
\]
where $\Pi=\{\tau\in\mathbb{R}:t+\tau\in\tilde{\mathbb{T}},\forall
t\in\tilde{\mathbb{T}}\}, \chi^{-1}: \tilde{\mathbb{T}}\rightarrow \mathbb{T}$ is the inverse mapping of $\chi$. $\varepsilon$ is called the $\varepsilon$-translation number of $f$.
\end{definition}
Obliviously, Definition \ref{final} greatly extends the concept of almost periodic functions on time scales that was introduced in \cite{bing5} and it unifies the concepts of almost periodic functions on time scales for the cases of $\mathbb{T}=\mathbb{R}, \mathbb{T}=\mathbb{Z}, \mathbb{T}=\overline{q^\mathbb{Z}}$ and the case $\mathbb{T}$ is an almost periodic time scale defined by Definition \ref{li67}.
Similarly, we can give new definitions of almost automorphic functions on time scales. These are our future goals.

\end{document}